\documentclass[final,3p,times]{elsarticle}

\usepackage{graphicx}%
\usepackage{multirow}%
\usepackage{amsmath,amssymb,amsfonts}%
\usepackage{amsthm}%
\usepackage{mathrsfs}%
\usepackage[title]{appendix}%
\usepackage{xcolor}%
\usepackage{textcomp}%
\usepackage{manyfoot}%
\usepackage{booktabs}%
\usepackage{algorithm}%
\usepackage{algorithmicx}%
\usepackage{algpseudocode}%
\usepackage{listings}%
\usepackage{xcolor}
\usepackage{subcaption}

\usepackage{tikz}
\usepackage{pgfplots}
\pgfplotsset{compat=newest}

\theoremstyle{definition}
\newtheorem{definition}{Definition}

\theoremstyle{plain}
\newtheorem{theorem}{Theorem}[section]
\newtheorem{lemma}{Lemma}[section]
\newtheorem{remark}{Remark}[section]

\numberwithin{equation}{section}

\usepackage[english]{babel}

\journal{Applied Numerical Mathematics}

\begin{document}

\begin{frontmatter}

\title{Generalized fractional  Laguerre orthogonal functions: projection and interpolation estimates}

\author[f4]{Mahmoud A. Zaky$^{*,}$}
  \ead{ma.zaky@yahoo.com;mibrahimm@imamu.edu.sa}

\address[f4]{Department of Mathematics and Statistics, College of Science, Imam Mohammad Ibn Saud Islamic University (IMSIU), Riyadh, Saudi Arabia}

\cortext[mycorrespondingauthor]{Corresponding author: M.A. Zaky}

\begin{abstract}
Classical Laguerre spectral approximations are highly effective on the
half-line when the target function is smooth in the usual polynomial scale.
However, their accuracy deteriorates for nonsmooth functions. Such behavior
appears naturally in fractional models, weakly singular integral equations, and
semi-infinite-domain approximations with limited regularity near the origin.
 The main contribution of this work is the construction and analysis of a
fractional Laguerre approximation framework tailored to nonsmooth functions on
the half-line.  We establish projection and interpolation error estimates in nonuniformly weighted Sobolev space. These estimates clarify how the fractional parameter adapts the approximation space to the regularity of nonsmooth functions and improves the resulting convergence behavior. We further introduce a generalized fractional Laguerre family with an additional algebraic parameter, which gives greater flexibility in controlling both the approximation space and the underlying weight. Numerical experiments confirm the theoretical estimates and demonstrate the advantage of the proposed functions over standard Laguerre-type approximations.
\end{abstract}

\end{frontmatter}

\section{Introduction}
Orthogonal spectral approximations on unbounded intervals are a fundamental tool in the numerical solution of differential equations, integral equations, and fractional models posed on semi-infinite domains. Among the classical orthogonal systems, Laguerre polynomials are particularly natural on \((0,\infty)\), since their orthogonality is associated with exponentially weighted spaces \cite{Szego1975}. Laguerre expansions and Laguerre-type spectral methods have therefore been widely used in approximation theory and scientific computing \cite{GottliebOrszag1977,liu2017spectral,Canuto2006}. When the underlying function is sufficiently smooth in the standard variable, these methods typically deliver rapid convergence and high accuracy. However, their performance may deteriorate significantly for nonsmooth functions, especially when the leading behavior near the origin is governed by algebraic or fractional powers. In this case, the function may become regular only after a suitable fractional change of variable, whereas its ordinary derivatives in \(x\) remain singular or weakly regular. Standard Laguerre expansions then approximate the function in a scale that does not reflect its intrinsic regularity, leading to slow decay of expansion coefficients and loss of spectral accuracy. This difficulty arises frequently in fractional differential equations, weakly singular Volterra and Fredholm integral equations, and models involving memory effects or nonlocal operators \cite{podlubny1998fractional,huang2018spectral,guo2007spectral,bhrawy2014new,chen2020efficient}. These observations motivate the construction of Laguerre-type approximation spaces that are adapted to fractional regularity rather than to the ordinary polynomial scale.

Several approaches have been developed to reduce the effect of weak
regularity. One direction is to use graded meshes or local refinement near the
singular region. This strategy is effective in finite difference and finite
element methods, but it does not fully preserve the global character of
spectral approximations \cite{stynes2022survey}. A second direction is to
enrich the approximation space by adding known singular functions
\cite{dell2024enriched,yang2026adaptive}. This can be highly accurate when the
dominant singular exponents are known in advance, but it may be less convenient
when the fractional structure is unknown or contains several components. A
third direction is to use mapped or nonclassical orthogonal systems. In this
approach, a nonlinear transformation modifies the approximation space and
allows the basis to better match the regularity of the target function
\cite{moussa2025mapped}.

Fractional and mapped orthogonal systems have also been used to improve
spectral approximations for problems with weak regularity \cite{q1,q2,q3}. Fractional Jacobi,
fractional Legendre, and fractional Laguerre functions have been developed for
fractional differential equations, variational problems, and related spectral
discretizations \cite{yu2019laguerre,bhrawy2014new}. These bases
incorporate fractional powers or nonlinear changes of variable into the
approximation space. More recent work on fractional and weakly singular
problems further confirms that adapting the basis to the regularity structure
is an effective way to recover high accuracy \cite{zaky2026new}. These
developments motivate the construction of new Laguerre-type systems adapted to
fractional scales on the half-line.

The aim of this paper is to construct and analyze a new class of fractional
Laguerre functions on \(\Lambda=(0,\infty)\). The proposed functions are
generated from generalized Laguerre polynomials through the fractional
variable \(x^\gamma\), where \(0<\gamma<1\). This construction preserves the
Laguerre structure on the half-line and produces approximation spaces spanned
by fractional monomials. The resulting basis is therefore naturally suited to
functions whose behavior is regular in the variable \(x^\gamma\).

We derive the main structural properties of the proposed functions, including
recurrence formulas, derivative identities, orthogonality, a
Sturm--Liouville equation, and a mapped Laguerre--Gauss quadrature rule. The
approximation theory is developed in weighted Sobolev-type spaces associated
with the fractional scale. In this setting, functions that have limited
regularity in the standard polynomial scale may possess higher regularity when
measured through the mapped derivative. This allows us to establish projection
and interpolation estimates that better reflect the behavior of functions with
fractional-power structure.

We also introduce a generalized fractional Laguerre family with an additional
algebraic scaling parameter. This extension provides more flexibility in
choosing the approximation space and the associated weight. It allows
algebraic factors to be incorporated directly into the basis while preserving
the underlying Laguerre structure. Projection and interpolation estimates are
proved for this generalized family as well.

The main contributions of this paper are as follows:
\begin{itemize}
\item A new class of fractional Laguerre functions on \((0,\infty)\) is
constructed using the fractional variable \(x^\gamma\).

\item Recurrence formulas, orthogonality relations, derivative identities, a
Sturm--Liouville equation, and a mapped Laguerre--Gauss quadrature rule are
derived for the proposed basis.

\item Projection estimates are established in weighted Sobolev-type spaces
adapted to the fractional scale.

\item A mapped interpolation operator is constructed, and stability,
interpolation error, and quadrature error estimates are proved.

\item A generalized fractional Laguerre family with an algebraic scaling
parameter is introduced and analyzed.
\end{itemize}

The paper is organized as follows. Section~\ref{sec:FLF} introduces the
fractional Laguerre functions and develops their structural properties.
Section~\ref{subsec:FLF-projection-estimate} establishes projection estimates,
while Section~\ref{subsec:FLF-interpolation-estimate} proves interpolation and
quadrature error estimates. Section~\ref{sec:GFLF} presents the generalized
fractional Laguerre functions and derives the corresponding approximation
theory. Numerical examples are included to illustrate the role of the
parameters and to verify the theoretical results.

\section{Fractional Laguerre functions}
\label{sec:FLF}

Let \(\Lambda=(0,\infty)\). This section introduces a scaled fractional Laguerre
system on \(\Lambda\) and records the identities required in the subsequent
projection, interpolation, and quadrature analysis. The construction is based
on the transformation
\begin{equation}\label{eq:FLF-map}
Y_{\beta,\gamma}(x):=(\beta+1)x^\gamma,
\qquad x>0,\qquad \beta>-1,\qquad 0<\gamma<1 .
\end{equation}
Then \(Y_{\beta,\gamma}:\Lambda\to\mathbb{R}^{+}\) is a one-to-one transformation,
and
\begin{equation}\label{eq:FLF-map-relations}
Y_{\beta,\gamma}=(\beta+1)x^\gamma,
\qquad
dY_{\beta,\gamma}=(\beta+1)\gamma x^{\gamma-1}\,dx,
\qquad
\partial_x=(\beta+1)\gamma x^{\gamma-1}\partial_{Y_{\beta,\gamma}} .
\end{equation}
For \(\theta>-1\), let \(L_m^{(\theta)}\) be the generalized Laguerre
polynomial of degree \(m\), satisfying
\begin{equation}\label{eq:FLF-Lag-orth}
\int_0^\infty
L_m^{(\theta)}(y)L_\ell^{(\theta)}(y)y^\theta e^{-y}\,dy
=
h_m^{(\theta)}\delta_{m\ell},
\qquad
h_m^{(\theta)}
=
\frac{\Gamma(m+\theta+1)}{\Gamma(m+1)} .
\end{equation}
The auxiliary Laguerre identities used below are recalled in Appendix~A.

\subsection{Definition and structural identities}
\label{subsec:FLF-def}

\begin{definition}[Scaled fractional Laguerre functions]\label{def:FLF}
For \(\theta>-1\), \(\beta>-1\), and \(0<\gamma<1\), define
\begin{equation}\label{eq:FLF-def}
\mathscr{L}_{m}^{(\theta,\beta,\gamma)}(x)
:=
L_m^{(\theta)}\!\left(Y_{\beta,\gamma}(x)\right)
=
L_m^{(\theta)}\!\left((\beta+1)x^\gamma\right),
\qquad m=0,1,\ldots .
\end{equation}
\end{definition}

The explicit representation of the scaled fractional Laguerre functions is
given by
\begin{equation}\label{eq:FLF-explicit-expanded}
\mathscr{L}_{m}^{(\theta,\gamma,\beta)}(x)
=
\sum_{r=0}^{m}
\frac{(-1)^r(\beta+1)^r}{r!}
\binom{m+\theta}{m-r}
x^{r\gamma},
\qquad x>0 .
\end{equation}
The following lemma shows how the classical Laguerre structure is inherited
by the scaled fractional Laguerre functions under the transformation
\eqref{eq:FLF-map}.

\begin{lemma}\label{lem:FLF-properties}
Let \(\theta>-1\), \(\beta>-1\), and \(0<\gamma<1\). The functions
\(\{\mathscr{L}_{m}^{(\theta,\beta,\gamma)}\}_{m\ge0}\) satisfy the following
properties.

\medskip
\noindent{\rm (i) Three-term recurrence.}
\begin{align}
\mathscr{L}_{0}^{(\theta,\beta,\gamma)}(x)&=1,
\qquad
\mathscr{L}_{1}^{(\theta,\beta,\gamma)}(x)
=
-(\beta+1)x^\gamma+\theta+1,
\nonumber\\
\mathscr{L}_{m+1}^{(\theta,\beta,\gamma)}(x)
&=
\frac{2m+\theta+1-(\beta+1)x^\gamma}{m+1}
\mathscr{L}_{m}^{(\theta,\beta,\gamma)}(x)
-
\frac{m+\theta}{m+1}
\mathscr{L}_{m-1}^{(\theta,\beta,\gamma)}(x),
\qquad m\ge1 .
\label{eq:FLF-recurrence}
\end{align}

\medskip
\noindent{\rm (ii) Mapped derivative.}
Define
\begin{equation}\label{eq:FLF-mapped-derivative}
\mathscr{D}_{\beta,\gamma}u(x)
:=
\frac{x^{1-\gamma}}{(\beta+1)\gamma}\partial_x u(x).
\end{equation}
Then
\begin{equation}\label{eq:FLF-deriv}
-\mathscr{D}_{\beta,\gamma}
\mathscr{L}_{m}^{(\theta,\beta,\gamma)}(x)
=
\mathscr{L}_{m-1}^{(\theta+1,\beta,\gamma)}(x)
=
\sum_{r=0}^{m-1}
\mathscr{L}_{r}^{(\theta,\beta,\gamma)}(x),
\qquad m\ge1 .
\end{equation}

\medskip
\noindent{\rm (iii) Orthogonality.}
Let
\begin{equation}\label{eq:FLF-weight}
\varrho^{\theta,\beta,\gamma}(x)
:=
(\beta+1)^{\theta+1}\gamma
x^{\gamma(\theta+1)-1}
e^{-(\beta+1)x^\gamma}.
\end{equation}
Then
\begin{equation}\label{eq:FLF-orth}
\int_0^\infty
\mathscr{L}_{m}^{(\theta,\beta,\gamma)}(x)
\mathscr{L}_{\ell}^{(\theta,\beta,\gamma)}(x)
\varrho^{\theta,\beta,\gamma}(x)\,dx
=
h_m^{(\theta)}\delta_{m\ell}.
\end{equation}

\medskip
\noindent{\rm (iv) Sturm--Liouville form.}
For \(m\ge0\),
\begin{equation}\label{eq:FLF-SL}
[\varrho^{\theta,\beta,\gamma}(x)]^{-1}
\partial_x
\left(
(\beta+1)^{\theta}
x^{\gamma\theta+1}
e^{-(\beta+1)x^\gamma}
\partial_x\mathscr{L}_{m}^{(\theta,\beta,\gamma)}(x)
\right)
+
m\gamma \mathscr{L}_{m}^{(\theta,\beta,\gamma)}(x)
=
0 .
\end{equation}

\medskip
\noindent{\rm (v) Scaled fractional Laguerre--Gauss quadrature.}
Let
\(\{y_i^{(\theta)},\varpi_i^{(\theta)}\}_{i=0}^{M}\) be the
Laguerre--Gauss nodes and weights associated with \(L_{M+1}^{(\theta)}\).
Define
\begin{equation}\label{eq:FLF-nodes}
x_i^{(\theta,\beta,\gamma)}
:=
\left(\frac{y_i^{(\theta)}}{\beta+1}\right)^{1/\gamma},
\qquad 0\le i\le M .
\end{equation}
Then
\begin{equation}\label{eq:FLF-quadrature}
\int_0^\infty
q(x)\varrho^{\theta,\beta,\gamma}(x)\,dx
=
\sum_{i=0}^{M}
q\!\left(x_i^{(\theta,\beta,\gamma)}\right)
\varpi_i^{(\theta)}
\end{equation}
for every
\begin{equation}\label{eq:FLF-space}
q\in\mathbb{P}_{2M+1}^{\gamma}
:=
\operatorname{span}
\left\{
1,x^\gamma,x^{2\gamma},\ldots,x^{(2M+1)\gamma}
\right\}.
\end{equation}
\end{lemma}

\begin{proof}
The recurrence \eqref{eq:FLF-recurrence} follows by substituting
\(y=(\beta+1)x^\gamma\) into the Laguerre recurrence relation.

For \eqref{eq:FLF-deriv}, the chain rule gives
\[
\mathscr{D}_{\beta,\gamma}
\mathscr{L}_{m}^{(\theta,\beta,\gamma)}(x)
=
\partial_y L_m^{(\theta)}(y)\big|_{y=(\beta+1)x^\gamma}.
\]
Using
\[
-\partial_y L_m^{(\theta)}(y)=L_{m-1}^{(\theta+1)}(y),
\]
and
\[
L_{m-1}^{(\theta+1)}(y)
=
\sum_{r=0}^{m-1}L_r^{(\theta)}(y),
\]
we obtain \eqref{eq:FLF-deriv}.

The substitution \(y=(\beta+1)x^\gamma\) yields
\[
y^\theta e^{-y}\,dy
=
\big((\beta+1)x^\gamma\big)^\theta
e^{-(\beta+1)x^\gamma}
(\beta+1)\gamma x^{\gamma-1}\,dx .
\]
Hence
\[
y^\theta e^{-y}\,dy
=
(\beta+1)^{\theta+1}\gamma
x^{\gamma(\theta+1)-1}
e^{-(\beta+1)x^\gamma}\,dx .
\]
Therefore,
\[
\begin{aligned}
&\int_0^\infty
\mathscr{L}_{m}^{(\theta,\beta,\gamma)}(x)
\mathscr{L}_{\ell}^{(\theta,\beta,\gamma)}(x)
\varrho^{\theta,\beta,\gamma}(x)\,dx
\\
&\qquad =
\int_0^\infty
L_m^{(\theta)}(y)L_\ell^{(\theta)}(y)y^\theta e^{-y}\,dy
=
h_m^{(\theta)}\delta_{m\ell},
\end{aligned}
\]
which proves \eqref{eq:FLF-orth}.

The Laguerre Sturm--Liouville equation is
\[
y^{-\theta}e^{y}
\partial_y
\left(
y^{\theta+1}e^{-y}\partial_y L_m^{(\theta)}(y)
\right)
+
mL_m^{(\theta)}(y)=0 .
\]
Using
\[
\partial_y=\frac{x^{1-\gamma}}{(\beta+1)\gamma}\partial_x,
\qquad
y=(\beta+1)x^\gamma,
\qquad
L_m^{(\theta)}(y)=\mathscr{L}_{m}^{(\theta,\beta,\gamma)}(x),
\]
gives \eqref{eq:FLF-SL}.

Finally,
\[
\int_0^\infty
q(x)\varrho^{\theta,\beta,\gamma}(x)\,dx
=
\int_0^\infty
q\!\left(\left(\frac{y}{\beta+1}\right)^{1/\gamma}\right)
y^\theta e^{-y}\,dy .
\]
If \(q\in\mathbb{P}_{2M+1}^{\beta,\gamma}\), then
\(q\!\left((y/(\beta+1))^{1/\gamma}\right)\) is a polynomial in \(y\)
of degree at most \(2M+1\). The Laguerre--Gauss rule is therefore exact, and
\eqref{eq:FLF-quadrature} follows.
\end{proof}

\subsection{Projection estimate}
\label{subsec:FLF-projection-estimate}

Let \(\theta>-1\), \(\beta>-1\), \(0<\gamma<1\), and let
\[
\varrho^{\theta,\beta,\gamma}(x)
:=
(\beta+1)^{\theta+1}\gamma
x^{\gamma(\theta+1)-1}e^{-(\beta+1)x^\gamma},
\qquad x\in\Lambda:=(0,\infty).
\]
For \(u\in L^2_{\varrho^{\theta,\beta,\gamma}}(\Lambda)\), we denote by
\(\Pi_M^{\theta,\beta,\gamma}u\) its weighted orthogonal projection onto
\(\mathbb{P}_{M}^{\gamma}\), namely
\begin{equation}\label{eq:FLF-projection-def}
\bigl(u-\Pi_M^{\theta,\beta,\gamma}u,w\bigr)_{\varrho^{\theta,\beta,\gamma}}
=
\int_0^\infty
\bigl(u-\Pi_M^{\theta,\beta,\gamma}u\bigr)(x)
w(x)\varrho^{\theta,\beta,\gamma}(x)\,dx
=
0,
\qquad
\forall w\in \mathbb{P}_{M}^{\gamma},
\end{equation}
where
\[
\mathbb{P}_{M}^{\gamma}
:=
\operatorname{span}
\{1,x^\gamma,x^{2\gamma},\ldots,x^{M\gamma}\}.
\]
Using the orthogonality relation \eqref{eq:FLF-orth}, the projection has the
modal representation
\begin{equation}\label{eq:FLF-projection-expansion}
\Pi_M^{\theta,\beta,\gamma}u
=
\sum_{m=0}^{M}
\widehat u_m^{\theta,\beta,\gamma}
\mathscr{L}_{m}^{(\theta,\beta,\gamma)},
\qquad
\widehat u_m^{\theta,\beta,\gamma}
=
\bigl(h_m^{(\theta)}\bigr)^{-1}
\int_0^\infty
u(x)\mathscr{L}_{m}^{(\theta,\beta,\gamma)}(x)
\varrho^{\theta,\beta,\gamma}(x)\,dx .
\end{equation}

The differentiation scale associated with the scaled fractional transformation
\(Y_{\beta,\gamma}(x)=(\beta+1)x^\gamma\) is generated by
\begin{equation}\label{eq:FLF-mapped-derivative-proj}
\mathscr{D}_{\beta,\gamma}u(x)
:=
\frac{x^{1-\gamma}}{(\beta+1)\gamma}\partial_x u(x).
\end{equation}
For an integer \(\mu\ge0\), we introduce the weighted Sobolev-type space
\begin{equation}\label{eq:FLF-weighted-space}
\mathcal{A}_{\theta,\beta,\gamma}^{\mu}(\Lambda)
:=
\left\{
v\in L^2_{\varrho^{\theta,\beta,\gamma}}(\Lambda):
\mathscr{D}_{\beta,\gamma}^{\,r}v
\in L^2_{\varrho^{\theta+r,\beta,\gamma}}(\Lambda),
\quad 1\le r\le \mu
\right\}.
\end{equation}
The associated seminorms and norm are
\[
|v|_{\mathcal{A}_{\theta,\beta,\gamma}^{r}}
:=
\|\mathscr{D}_{\beta,\gamma}^{\,r}v\|_{\varrho^{\theta+r,\beta,\gamma}},
\qquad
\|v\|_{\mathcal{A}_{\theta,\beta,\gamma}^{\mu}}
:=
\left(
\sum_{r=0}^{\mu}
|v|_{\mathcal{A}_{\theta,\beta,\gamma}^{r}}^2
\right)^{1/2}.
\]

\begin{theorem}\label{thm:FLF-projection-estimate}
Let \(\mu,M\in\mathbb{N}\), \(0\le s\le \widehat\mu\), where 
$\widehat\mu:=\min\{\mu,M+1\},$ 
and let \(\theta>-1\), \(\beta>-1\), \(0<\gamma<1\). If
\(u\in\mathcal{A}_{\theta,\beta,\gamma}^{\mu}(\Lambda)\), then
\begin{equation}\label{eq:FLF-projection-estimate}
\left\|
\mathscr{D}_{\beta,\gamma}^{\,s}
\bigl(u-\Pi_M^{\theta,\beta,\gamma}u\bigr)
\right\|_{\varrho^{\theta+s,\beta,\gamma}}
\le
\left[
\frac{(M-\widehat\mu+1)!}{(M-s+1)!}
\right]^{1/2}
\left\|
\mathscr{D}_{\beta,\gamma}^{\,\widehat\mu}u
\right\|_{\varrho^{\theta+\widehat\mu,\beta,\gamma}} .
\end{equation}
In particular, for \(\theta=s=0\) and \(\mu<M+1\),
\begin{equation}\label{eq:FLF-projection-special}
\|u-\Pi_M u\|_{\varrho^{0,\beta,\gamma}}
\le
C M^{-\mu/2}
\left\|
\mathscr{D}_{\beta,\gamma}^{\,\mu}u
\right\|_{\varrho^{\mu,\beta,\gamma}},
\end{equation}
where \(\Pi_M=\Pi_M^{0,\beta,\gamma}\).
\end{theorem}

\begin{proof}
By \eqref{eq:FLF-deriv} and the definition of
\(\mathscr{D}_{\beta,\gamma}\), repeated application gives
\begin{equation}\label{eq:FLF-derivative-power}
\mathscr{D}_{\beta,\gamma}^{\,r}
\mathscr{L}_{m}^{(\theta,\beta,\gamma)}(x)
=
(-1)^r
\mathscr{L}_{m-r}^{(\theta+r,\beta,\gamma)}(x),
\qquad 0\le r\le m .
\end{equation}
The sign has no effect on the squared norms. Hence, if
\[
u(x)
=
\sum_{m=0}^{\infty}
\widehat u_m^{\theta,\beta,\gamma}
\mathscr{L}_{m}^{(\theta,\beta,\gamma)}(x),
\]
then the orthogonality relation \eqref{eq:FLF-orth} gives, for \(r\ge1\),
\[
\left\|
\mathscr{D}_{\beta,\gamma}^{\,r}u
\right\|_{\varrho^{\theta+r,\beta,\gamma}}^2
=
\sum_{m=r}^{\infty}
h_{m-r}^{(\theta+r)}
\left|\widehat u_m^{\theta,\beta,\gamma}\right|^2 .
\]
Therefore,
\[
\begin{aligned}
&
\left\|
\mathscr{D}_{\beta,\gamma}^{\,s}
\bigl(u-\Pi_M^{\theta,\beta,\gamma}u\bigr)
\right\|_{\varrho^{\theta+s,\beta,\gamma}}^2
\\
&\qquad =
\sum_{m=M+1}^{\infty}
h_{m-s}^{(\theta+s)}
\left|\widehat u_m^{\theta,\beta,\gamma}\right|^2
\\
&\qquad \le
\max_{m\ge M+1}
\frac{
h_{m-s}^{(\theta+s)}
}{
h_{m-\widehat\mu}^{(\theta+\widehat\mu)}
}
\sum_{m=M+1}^{\infty}
h_{m-\widehat\mu}^{(\theta+\widehat\mu)}
\left|\widehat u_m^{\theta,\beta,\gamma}\right|^2
\\
&\qquad \le
\frac{
h_{M+1-s}^{(\theta+s)}
}{
h_{M+1-\widehat\mu}^{(\theta+\widehat\mu)}
}
\left\|
\mathscr{D}_{\beta,\gamma}^{\,\widehat\mu}u
\right\|_{\varrho^{\theta+\widehat\mu,\beta,\gamma}}^2 .
\end{aligned}
\]
Using
\[
h_m^{(\theta)}
=
\frac{\Gamma(m+\theta+1)}{\Gamma(m+1)},
\]
we obtain
\[
\frac{
h_{M+1-s}^{(\theta+s)}
}{
h_{M+1-\widehat\mu}^{(\theta+\widehat\mu)}
}
=
\frac{(M-\widehat\mu+1)!}{(M-s+1)!}.
\]
Substitution into the preceding inequality proves
\eqref{eq:FLF-projection-estimate}.

To derive the special case \eqref{eq:FLF-projection-special}, we use the
standard gamma-ratio estimate: for \(\xi,\zeta\in\mathbb{R}\),
\(\kappa\in\mathbb{N}\), \(\kappa+\xi>1\), and \(\kappa+\zeta>1\),
\begin{equation}\label{eq:FLF-gamma-ratio-bound}
\frac{\Gamma(\kappa+\xi)}{\Gamma(\kappa+\zeta)}
\le
\mathfrak{c}_{\kappa}^{\xi,\zeta}\kappa^{\xi-\zeta},
\end{equation}
where
\begin{equation}\label{eq:FLF-gamma-ratio-constant}
\mathfrak{c}_{\kappa}^{\xi,\zeta}
=
\exp
\left(
\frac{\xi-\zeta}{2(\kappa+\zeta-1)}
+
\frac{1}{12(\kappa+\xi-1)}
+
\frac{(\xi-\zeta)^2}{\kappa}
\right).
\end{equation}
Applying \eqref{eq:FLF-gamma-ratio-bound} to the factorial ratio in
\eqref{eq:FLF-projection-estimate} gives
\[
\frac{(M-\mu+1)!}{(M+1)!}
\le C M^{-\mu}.
\]
Consequently,
\[
\|u-\Pi_Mu\|_{\varrho^{0,\beta,\gamma}}
\le
C M^{-\mu/2}
\|\mathscr{D}_{\beta,\gamma}^{\,\mu}u\|_{\varrho^{\mu,\beta,\gamma}} .
\]
This proves \eqref{eq:FLF-projection-special}.
\end{proof}

\subsection{Interpolation estimate}
\label{subsec:FLF-interpolation-estimate}

Let \(\{x_i^{(\theta,\beta,\gamma)}\}_{i=0}^{M}\) be the scaled fractional
Laguerre--Gauss points defined in \eqref{eq:FLF-nodes}. For
\(0\le i\le M\), define the scaled fractional Lagrange functions by
\begin{equation}\label{eq:FLF-mapped-Lagrange}
\ell_i^{(\theta,\beta,\gamma)}\!\left(Y_{\beta,\gamma}(x)\right)
:=
\frac{
\displaystyle
\prod_{\substack{r=0\\ r\ne i}}^{M}
\left(
Y_{\beta,\gamma}(x)-Y_{\beta,\gamma}\!\left(x_r^{(\theta,\beta,\gamma)}\right)
\right)
}{
\displaystyle
\prod_{\substack{r=0\\ r\ne i}}^{M}
\left(
Y_{\beta,\gamma}\!\left(x_i^{(\theta,\beta,\gamma)}\right)
-
Y_{\beta,\gamma}\!\left(x_r^{(\theta,\beta,\gamma)}\right)
\right)
}.
\end{equation}
Since \(Y_{\beta,\gamma}(x)=(\beta+1)x^\gamma\) and 
$Y_{\beta,\gamma}\!\left(x_i^{(\theta,\beta,\gamma)}\right)
=
y_i^{(\theta)},$ 
this can be written as
\begin{equation}\label{eq:FLF-mapped-Lagrange-power-simplified}
\ell_i^{(\theta,\beta,\gamma)}\!\left(Y_{\beta,\gamma}(x)\right)
=
\frac{
\displaystyle
\prod_{\substack{r=0\\ r\ne i}}^{M}
\left(
x^\gamma-
\left(x_r^{(\theta,\beta,\gamma)}\right)^\gamma
\right)
}{
\displaystyle
\prod_{\substack{r=0\\ r\ne i}}^{M}
\left(
\left(x_i^{(\theta,\beta,\gamma)}\right)^\gamma
-
\left(x_r^{(\theta,\beta,\gamma)}\right)^\gamma
\right)
}.
\end{equation}
The interpolation operator
$\mathcal{I}_{M}^{\theta,\beta,\gamma}:C(\Lambda)\longrightarrow
\mathbb{P}_{M}^{\gamma}$
is defined by
\begin{equation}\label{eq:FLF-interpolation-operator}
\mathcal{I}_{M}^{\theta,\beta,\gamma}v(x)
=
\sum_{i=0}^{M}
v\!\left(x_i^{(\theta,\beta,\gamma)}\right)
\ell_i^{(\theta,\beta,\gamma)}\!\left(Y_{\beta,\gamma}(x)\right).
\end{equation}
Thus
$\mathcal{I}_{M}^{\theta,\beta,\gamma}v
\!\left(x_i^{(\theta,\beta,\gamma)}\right)
=
v\!\left(x_i^{(\theta,\beta,\gamma)}\right),
\qquad 0\le i\le M .$

We first state a stability estimate.

\begin{theorem}\label{thm:FLF-interpolation-stability}
Let \(\theta,\ \beta>-1\), and \(\gamma \in (0,1)\). If
\(v\in C(\Lambda)\cap\mathcal{A}_{\theta,\beta,\gamma}^{1}(\Lambda)\) and
\(\mathscr{D}_{\beta,\gamma}v\in L^2_{\varrho^{\theta,\beta,\gamma}}(\Lambda)\),
then
\begin{equation}\label{eq:FLF-interpolation-stability}
\left\|
\mathcal{I}_{M}^{\theta,\beta,\gamma}v
\right\|_{\varrho^{\theta,\beta,\gamma}}
\le
C
\left[
M^{-1/2}
\left\|
\mathscr{D}_{\beta,\gamma}v
\right\|_{\varrho^{\theta,\beta,\gamma}}
+
2\sqrt{\log M}
\left\|
v
\right\|_{\mathcal{A}_{\theta,\beta,\gamma}^{1}}
\right].
\end{equation}
\end{theorem}

\begin{proof}
Set
\[
x(y)=\left(\frac{y}{\beta+1}\right)^{1/\gamma},
\qquad
\widetilde v(y)=v(x(y)).
\]
Then \eqref{eq:FLF-interpolation-operator} is the pullback of the
Laguerre--Gauss interpolation operator:
\[
\mathcal{I}_{M}^{\theta,\beta,\gamma}v(x)
=
\mathcal{I}_{M}^{\theta}\widetilde v(y)
:=
\sum_{i=0}^{M}
\widetilde v\!\left(y_i^{(\theta)}\right)\ell_i(y),
\qquad y=Y_{\beta,\gamma}(x).
\]
The Laguerre interpolation stability estimate gives
\[
\left\|
\mathcal{I}_{M}^{\theta}\widetilde v
\right\|_{y^\theta e^{-y}}
\le
C\left(
M^{-1/2}\mathcal{M}_1(\widetilde v)^{1/2}
+
2\sqrt{\log M}\,\mathcal{M}_2(\widetilde v)^{1/2}
\right),
\]
where
\[
\mathcal{M}_1(\widetilde v)
=
\int_0^\infty
(\partial_y\widetilde v(y))^2y^\theta e^{-y}\,dy,
\]
and
\[
\mathcal{M}_2(\widetilde v)
=
\int_0^\infty
\left(
\widetilde v^2
+
y(\partial_y\widetilde v)^2
\right)y^\theta e^{-y}\,dy .
\]

The transformation \(y=(\beta+1)x^\gamma\) gives
\[
dy=(\beta+1)\gamma x^{\gamma-1}\,dx,
\qquad
e^{-y}=e^{-(\beta+1)x^\gamma},
\]
and
\[
\partial_y\widetilde v(y)
=
\frac{x^{1-\gamma}}{(\beta+1)\gamma}\partial_xv(x)
=
\mathscr{D}_{\beta,\gamma}v(x).
\]
Consequently,
\[
\mathcal{M}_1(\widetilde v)
=
\int_0^\infty
\left(\mathscr{D}_{\beta,\gamma}v(x)\right)^2
\varrho^{\theta,\beta,\gamma}(x)\,dx
=
\left\|
\mathscr{D}_{\beta,\gamma}v
\right\|_{\varrho^{\theta,\beta,\gamma}}^2 .
\]
Similarly,
\[
\begin{aligned}
\mathcal{M}_2(\widetilde v)
&=
\int_0^\infty
\left[
v^2+
(\beta+1)x^\gamma\left(\mathscr{D}_{\beta,\gamma}v\right)^2
\right]
\varrho^{\theta,\beta,\gamma}(x)\,dx
\\
&=
\|v\|_{\varrho^{\theta,\beta,\gamma}}^2
+
\|\mathscr{D}_{\beta,\gamma}v\|_{\varrho^{\theta+1,\beta,\gamma}}^2
\le
\left\|
v
\right\|_{\mathcal{A}_{\theta,\beta,\gamma}^{1}}^2 .
\end{aligned}
\]
Combining the last two identities with the Laguerre interpolation stability
bound proves \eqref{eq:FLF-interpolation-stability}.
\end{proof}

The next theorem gives the corresponding interpolation error estimate.

\begin{theorem}\label{thm:FLF-interpolation-error}
Let \(\mu,M\in\mathbb{N}\), \(\theta>-1\), \(\beta>-1\), \(0<\gamma<1\), and
$\widehat\mu:=\min\{\mu,M+1\}.$

If \(v\in C(\Lambda)\cap \mathcal{A}_{\theta,\beta,\gamma}^{\mu}(\Lambda)\) and
\(\mathscr{D}_{\beta,\gamma}v\in
\mathcal{A}_{\theta,\beta,\gamma}^{\mu-1}(\Lambda)\), then
\begin{equation}\label{eq:FLF-interpolation-error}
\left\|
\mathcal{I}_{M}^{\theta,\beta,\gamma}v-v
\right\|_{\varrho^{\theta,\beta,\gamma}}
\le
C
\left[
\frac{(M+1-\widehat\mu)!}
{M!}
\right]^{1/2}
\left[
\left\|
\mathscr{D}_{\beta,\gamma}^{\,\widehat\mu}v
\right\|_{\varrho^{\theta+\mu-1,\beta,\gamma}}
+
2\sqrt{\log M}
\left\|
\mathscr{D}_{\beta,\gamma}^{\,\widehat\mu}v
\right\|_{\varrho^{\theta+\mu,\beta,\gamma}}
\right].
\end{equation}
\end{theorem}

\begin{proof}
By the triangle inequality,
\begin{equation}\label{eq:FLF-interpolation-proof-split}
\left\|
\mathcal{I}_{M}^{\theta,\beta,\gamma}v-v
\right\|_{\varrho^{\theta,\beta,\gamma}}
\le
\left\|
\mathcal{I}_{M}^{\theta,\beta,\gamma}v
-
\Pi_M^{\theta,\beta,\gamma}v
\right\|_{\varrho^{\theta,\beta,\gamma}}
+
\left\|
\Pi_M^{\theta,\beta,\gamma}v-v
\right\|_{\varrho^{\theta,\beta,\gamma}} .
\end{equation}
The second term is bounded by Theorem~\ref{thm:FLF-projection-estimate}. For
the first term, the interpolation stability estimate gives
\[
\begin{aligned}
&
\left\|
\mathcal{I}_{M}^{\theta,\beta,\gamma}v
-
\Pi_M^{\theta,\beta,\gamma}v
\right\|_{\varrho^{\theta,\beta,\gamma}}
\\
&\qquad =
\left\|
\mathcal{I}_{M}^{\theta,\beta,\gamma}
\left(v-\Pi_M^{\theta,\beta,\gamma}v\right)
\right\|_{\varrho^{\theta,\beta,\gamma}}
\\
&\qquad \le
C
\left[
M^{-1/2}
\left\|
\mathscr{D}_{\beta,\gamma}
\left(v-\Pi_M^{\theta,\beta,\gamma}v\right)
\right\|_{\varrho^{\theta,\beta,\gamma}}
+
2\sqrt{\log M}
\left\|
v-\Pi_M^{\theta,\beta,\gamma}v
\right\|_{\mathcal{A}_{\theta,\beta,\gamma}^{1}}
\right].
\end{aligned}
\]
The terms involving
\(\|v-\Pi_M^{\theta,\beta,\gamma}v\|_{\mathcal{A}_{\theta,\beta,\gamma}^{r}}\),
\(r=0,1\), are controlled by Theorem~\ref{thm:FLF-projection-estimate}.
It remains to estimate the derivative term in the same weighted space. We write
\begin{equation}\label{eq:FLF-derivative-projection-split}
\begin{aligned}
&
\left\|
\mathscr{D}_{\beta,\gamma}
\left(v-\Pi_M^{\theta,\beta,\gamma}v\right)
\right\|_{\varrho^{\theta,\beta,\gamma}}
\\
&\qquad \le
\left\|
\mathscr{D}_{\beta,\gamma}v
-
\Pi_M^{\theta,\beta,\gamma}\{\mathscr{D}_{\beta,\gamma}v\}
\right\|_{\varrho^{\theta,\beta,\gamma}}
+
\left\|
\Pi_M^{\theta,\beta,\gamma}\{\mathscr{D}_{\beta,\gamma}v\}
-
\mathscr{D}_{\beta,\gamma}\{\Pi_M^{\theta,\beta,\gamma}v\}
\right\|_{\varrho^{\theta,\beta,\gamma}} .
\end{aligned}
\end{equation}

Let
\[
v(x)=
\sum_{m=0}^{\infty}
\widehat v_m^{\theta,\beta,\gamma}
\mathscr{L}_{m}^{(\theta,\beta,\gamma)}(x).
\]
Using \eqref{eq:FLF-deriv}, we obtain
\[
\begin{aligned}
\mathscr{D}_{\beta,\gamma}v(x)
&=
-\sum_{m=1}^{\infty}
\widehat v_m^{\theta,\beta,\gamma}
\sum_{r=0}^{m-1}
\mathscr{L}_{r}^{(\theta,\beta,\gamma)}(x)
\\
&=
-\sum_{r=0}^{\infty}
\left[
\sum_{m=r+1}^{\infty}
\widehat v_m^{\theta,\beta,\gamma}
\right]
\mathscr{L}_{r}^{(\theta,\beta,\gamma)}(x).
\end{aligned}
\]
Hence
\[
\Pi_M^{\theta,\beta,\gamma}\{\mathscr{D}_{\beta,\gamma}v\}(x)
=
\sum_{r=0}^{M}
\widehat v_{1,r}^{\theta,\beta,\gamma}
\mathscr{L}_{r}^{(\theta,\beta,\gamma)}(x),
\]
where
\[
\widehat v_{1,r}^{\theta,\beta,\gamma}
:=
-
\sum_{m=r+1}^{\infty}
\widehat v_m^{\theta,\beta,\gamma}.
\]
Similarly,
\[
\mathscr{D}_{\beta,\gamma}\{\Pi_M^{\theta,\beta,\gamma}v\}(x)
=
\sum_{r=0}^{M-1}
\left(
\widehat v_{1,r}^{\theta,\beta,\gamma}
-
\widehat v_{1,M}^{\theta,\beta,\gamma}
\right)
\mathscr{L}_{r}^{(\theta,\beta,\gamma)}(x).
\]
Therefore,
\begin{equation}\label{eq:FLF-projection-commutator-bound}
\begin{aligned}
&
\left\|
\Pi_M^{\theta,\beta,\gamma}\{\mathscr{D}_{\beta,\gamma}v\}
-
\mathscr{D}_{\beta,\gamma}\{\Pi_M^{\theta,\beta,\gamma}v\}
\right\|_{\varrho^{\theta,\beta,\gamma}}^2
\\
&\qquad =
\sum_{r=0}^{M}
h_r^{(\theta)}
\left(
\widehat v_{1,M}^{\theta,\beta,\gamma}
\right)^2
\\
&\qquad =
h_M^{(\theta)}
\left(
\widehat v_{1,M}^{\theta,\beta,\gamma}
\right)^2
\sum_{r=0}^{M}
\frac{h_r^{(\theta)}}{h_M^{(\theta)}}
\\
&\qquad \le
\left\|
\mathscr{D}_{\beta,\gamma}v
-
\Pi_{M-1}^{\theta,\beta,\gamma}\{\mathscr{D}_{\beta,\gamma}v\}
\right\|_{\varrho^{\theta,\beta,\gamma}}^2
\sum_{r=0}^{M}
\frac{h_r^{(\theta)}}{h_M^{(\theta)}} .
\end{aligned}
\end{equation}
It remains to bound
\[
\mathfrak{s}_M
:=
\sum_{r=0}^{M}
\frac{h_r^{(\theta)}}{h_M^{(\theta)}} .
\]
If \(\theta\ge0\), then \(h_r^{(\theta)}\le h_M^{(\theta)}\) for \(0\le r\le M\),
and hence \(\mathfrak{s}_M\le M+1\). If \(-1<\theta<0\), Stirling's formula
gives, for sufficiently large \(r\le M\),
\[
\frac{h_r^{(\theta)}}{h_M^{(\theta)}}
=
\frac{\Gamma(M+1)\Gamma(r+\theta+1)}
{\Gamma(M+\theta+1)\Gamma(r+1)}
\sim
M^{-\theta}r^\theta .
\]
Consequently,
\[
\mathfrak{s}_M
\le
M^{-\theta}
\left(
C+
C\sum_{r=1}^{M}r^\theta
\right)
\le CM .
\]
Combining this bound with
\eqref{eq:FLF-interpolation-proof-split}--\eqref{eq:FLF-projection-commutator-bound}
and Theorem~\ref{thm:FLF-projection-estimate} yields
\eqref{eq:FLF-interpolation-error}.
\end{proof}

\begin{remark}\label{rem:FLF-quadrature-error}
Let
\(\{x_i^{(\theta,\beta,\gamma)}\}_{i=0}^{M}\) be the scaled fractional
Laguerre--Gauss nodes defined in \eqref{eq:FLF-nodes}. Then
\begin{equation}\label{eq:FLF-quadrature-error}
\left|
\int_0^\infty
v(x)\varrho^{\theta,\beta,\gamma}(x)\,dx
-
\sum_{i=0}^{M}
v\!\left(x_i^{(\theta,\beta,\gamma)}\right)
\varpi_i^{(\theta)}
\right|
\le
\sqrt{\Gamma(\theta+1)}
\left\|
\mathcal{I}_{M}^{\theta,\beta,\gamma}v-v
\right\|_{\varrho^{\theta,\beta,\gamma}} .
\end{equation}
Indeed,
\[
\sum_{i=0}^{M}
v\!\left(x_i^{(\theta,\beta,\gamma)}\right)
\varpi_i^{(\theta)}
=
\int_0^\infty
\mathcal{I}_{M}^{\theta,\beta,\gamma}v(x)
\varrho^{\theta,\beta,\gamma}(x)\,dx,
\]
while
\[
\int_0^\infty
\varrho^{\theta,\beta,\gamma}(x)\,dx
=
\int_0^\infty y^\theta e^{-y}\,dy
=
\Gamma(\theta+1).
\]
The estimate follows from the Cauchy--Schwarz inequality.
\end{remark}

\section{Generalized fractional Laguerre functions}
\label{sec:GFLF}

The fractional Laguerre functions introduced in the preceding section provide a
basis generated by the fractional variable \(x^\gamma\). For some approximation
problems, however, it is useful to modify the algebraic behavior of the basis
without changing the Laguerre variable. This motivates the following
generalized family. The additional parameter changes the algebraic scaling and
allows the approximation space and the weight to be tuned independently.

\subsection{Definition and structural properties}
\label{subsec:GFLF-def}

\begin{definition}[Generalized fractional Laguerre functions]
\label{def:GFLF}
Let \(\theta>-1\), \(\sigma\in\mathbb{R}\), \(\beta>-1\), and \(0<\gamma<1\). The generalized
fractional Laguerre functions are defined by
\begin{equation}\label{eq:GFLF-def}
\mathscr{L}_{m}^{(\theta,\sigma,\gamma,\beta)}(x)
:=
x^{\frac{\gamma(\theta-\sigma)}{2}}
\mathscr{L}_{m}^{(\theta,\gamma,\beta)}(x),
\qquad m=0,1,\ldots .
\end{equation}
Here
\[
\mathscr{L}_{m}^{(\theta,\gamma,\beta)}(x)
:=
L_m^{(\theta)}\!\left((\beta+1)x^\gamma\right).
\]
Equivalently,
\[
\mathscr{L}_{m}^{(\theta,\sigma,\gamma,\beta)}(x)
=
x^{\frac{\gamma(\theta-\sigma)}{2}}
L_m^{(\theta)}\!\left((\beta+1)x^\gamma\right).
\]
In particular,
\[
\mathscr{L}_{m}^{(\theta,\theta,\gamma,\beta)}(x)
=
\mathscr{L}_{m}^{(\theta,\gamma,\beta)}(x).
\]
\end{definition}

The generalized scaled weight is defined by
\[
\varrho^{\theta,\sigma,\gamma,\beta}(x)
:=
\gamma(\beta+1)^{\theta+1}
x^{\gamma(\sigma+1)-1}e^{-(\beta+1)x^\gamma},
\qquad x\in(0,\infty).
\]
Then the orthogonality relation reads
\begin{equation}\label{eq:GFLF-orth}
\int_0^\infty
\mathscr{L}_{m}^{(\theta,\sigma,\gamma,\beta)}(x)
\mathscr{L}_{\ell}^{(\theta,\sigma,\gamma,\beta)}(x)
\varrho^{\theta,\sigma,\gamma,\beta}(x)\,dx
=
h_m^{(\theta)}\delta_{m\ell},
\end{equation}
where
\begin{equation}\label{eq:GFLF-norm}
h_m^{(\theta)}
=
\frac{\Gamma(m+\theta+1)}{\Gamma(m+1)} .
\end{equation}

The ordinary derivative of the generalized basis follows from
\eqref{eq:GFLF-def} and the derivative identity for \(L_m^{(\theta)}\). One
obtains
\begin{align}
\partial_x\mathscr{L}_{m}^{(\theta,\sigma,\gamma,\beta)}(x)
&=
\frac{\gamma(\theta-\sigma)}{2}
x^{\frac{\gamma(\theta-\sigma)}{2}-1}
\mathscr{L}_{m}^{(\theta,\gamma,\beta)}(x)
+
x^{\frac{\gamma(\theta-\sigma)}{2}}
\partial_x\mathscr{L}_{m}^{(\theta,\gamma,\beta)}(x)
\nonumber\\
&=
\frac{\gamma(\theta-\sigma)}{2}
x^{-1}
\mathscr{L}_{m}^{(\theta,\sigma,\gamma,\beta)}(x)
-
\gamma(\beta+1)x^{\gamma-1}
\mathscr{L}_{m-1}^{(\theta+1,\sigma+1,\gamma,\beta)}(x),
\qquad m\ge1 .
\label{eq:GFLF-ordinary-derivative}
\end{align}

For the approximation analysis, it is convenient to introduce the scaled
mapped derivative
\begin{equation}\label{eq:GFLF-scaled-derivative}
\mathscr{D}_{\theta,\sigma,\gamma,\beta}u
:=
x^{\frac{\gamma(\theta-\sigma)}{2}}
\mathscr{D}_{\gamma,\beta}
\left(
x^{-\frac{\gamma(\theta-\sigma)}{2}}u
\right),
\qquad
\mathscr{D}_{\gamma,\beta}u
:=
\frac{x^{1-\gamma}}{\gamma(\beta+1)}\partial_x u .
\end{equation}
Since \(\mathscr{D}_{\gamma,\beta}=d/dy\) with
\(y=(\beta+1)x^\gamma\), we have
\begin{equation}\label{eq:GFLF-important-derivative}
-\mathscr{D}_{\theta,\sigma,\gamma,\beta}
\mathscr{L}_{m}^{(\theta,\sigma,\gamma,\beta)}(x)
=
\mathscr{L}_{m-1}^{(\theta+1,\sigma+1,\gamma,\beta)}(x)
=
\sum_{r=0}^{m-1}
\mathscr{L}_{r}^{(\theta,\sigma,\gamma,\beta)}(x),
\qquad m\ge1 .
\end{equation}
Thus the generalized basis preserves a closed degree-lowering differentiation
property in the scaled fractional variable.

Let
\(\{y_i^{(\theta)}\}_{i=0}^{M}\) and
\(\{\varpi_i^{(\theta)}\}_{i=0}^{M}\) be the Laguerre--Gauss nodes and weights
associated with the classical Laguerre weight \(y^\theta e^{-y}\) on
\((0,\infty)\). For the generalized weighted rule, set
\begin{equation}\label{eq:GFLF-nodes-weights}
x_i^{(\theta,\sigma,\gamma,\beta)}
:=
\left(\frac{y_i^{(\theta)}}{\beta+1}\right)^{1/\gamma},
\qquad
\lambda_i^{(\theta,\sigma,\gamma,\beta)}
:=
\left(x_i^{(\theta,\sigma,\gamma,\beta)}\right)^{\gamma(\sigma-\theta)}
\varpi_i^{(\theta)},
\qquad 0\le i\le M .
\end{equation}
Equivalently,
\[
\lambda_i^{(\theta,\sigma,\gamma,\beta)}
=
(\beta+1)^{\theta-\sigma}
\left(y_i^{(\theta)}\right)^{\sigma-\theta}
\varpi_i^{(\theta)} .
\]
For \(\eta\in\mathbb{R}\), define
\begin{equation}\label{eq:GFLF-space}
\mathbb{P}_{M}^{\eta,\gamma}
:=
\left\{
x^\eta q\!\left(x^\gamma\right):
q\in\mathbb{P}_{M}
\right\}.
\end{equation}
The generalized scaled quadrature formula is
\begin{equation}\label{eq:GFLF-quadrature}
\int_0^\infty
f(x)\varrho^{\theta,\sigma,\gamma,\beta}(x)\,dx
=
\sum_{i=0}^{M}
f\!\left(x_i^{(\theta,\sigma,\gamma,\beta)}\right)
\lambda_i^{(\theta,\sigma,\gamma,\beta)},
\qquad
\forall f\in
\mathbb{P}_{2M+1}^{\gamma(\theta-\sigma),\gamma}.
\end{equation}
Indeed, if
\[
f(x)=x^{\gamma(\theta-\sigma)}q\!\left(x^\gamma\right),
\qquad q\in\mathbb{P}_{2M+1},
\]
then, under the change of variable \(y=(\beta+1)x^\gamma\),
\[
f(x)\varrho^{\theta,\sigma,\gamma,\beta}(x)\,dx
=
q\!\left(\frac{y}{\beta+1}\right)y^\theta e^{-y}\,dy.
\]
Since \(q(y/(\beta+1))\) is still a polynomial of degree at most \(2M+1\) in
\(y\), the exactness follows from the classical Laguerre--Gauss quadrature
rule.

Finally, the explicit representation follows from the closed form of the
Laguerre polynomial:
\begin{equation}\label{eq:GFLF-explicit}
\mathscr{L}_{m}^{(\theta,\sigma,\gamma,\beta)}(x)
=
\sum_{r=0}^{m}
\frac{(-1)^r(\beta+1)^r}{r!}
\binom{m+\theta}{m-r}
x^{\frac{\gamma(\theta-\sigma)}{2}+r\gamma}. ,
\qquad x>0 .
\end{equation}

This representation shows that the generalized system belongs to the algebraic
fractional space generated by powers of \(x^\gamma\). The scale parameter
\(\beta\) modifies the coefficients and the associated weighted structure, but
not the underlying fractional polynomial span.

\subsection{Projection estimate}
\label{subsec:GFLF-projection}

Let \(\theta>-1\), \(\sigma\in\mathbb{R}\), \(\beta>-1\), and \(0<\gamma<1\).
Set
\[
\eta_{\theta,\sigma,\gamma}:=\frac{\gamma(\theta-\sigma)}{2}.
\]
Recall that
\[
\varrho^{\theta,\sigma,\gamma,\beta}(x)
=
\gamma(\beta+1)^{\theta+1}
x^{\gamma(\sigma+1)-1}e^{-(\beta+1)x^\gamma},
\qquad x\in\Lambda:=(0,\infty).
\]
We define the weighted orthogonal projection
\[
\Pi_M^{\theta,\sigma,\gamma,\beta}:
L^2_{\varrho^{\theta,\sigma,\gamma,\beta}}(\Lambda)
\longrightarrow
\mathbb{P}_{M}^{\eta_{\theta,\sigma,\gamma},\gamma}
\]
by
\begin{equation}\label{eq:GFLF-projection-def}
\bigl(
u-\Pi_M^{\theta,\sigma,\gamma,\beta}u,w
\bigr)_{\varrho^{\theta,\sigma,\gamma,\beta}}
=
0,
\qquad
\forall w\in
\mathbb{P}_{M}^{\eta_{\theta,\sigma,\gamma},\gamma}.
\end{equation}
Here
\[
\mathbb{P}_{M}^{\eta_{\theta,\sigma,\gamma},\gamma}
:=
\left\{
x^{\eta_{\theta,\sigma,\gamma}}q(x):
q\in\mathbb{P}^\gamma_{M}
\right\},
\qquad
\mathbb{P}_{M}^{\gamma}
=
\operatorname{span}
\{1,x^\gamma,\ldots,x^{M\gamma}\}.
\]

By the orthogonality relation \eqref{eq:GFLF-orth}, the projection admits the
expansion
\begin{equation}\label{eq:GFLF-projection-expansion}
\Pi_M^{\theta,\sigma,\gamma,\beta}u
=
\sum_{m=0}^{M}
\widehat u_m^{\theta,\sigma,\gamma,\beta}
\mathscr{L}_{m}^{(\theta,\sigma,\gamma,\beta)} .
\end{equation}
The corresponding coefficients are given by
\begin{equation}\label{eq:GFLF-coeff}
\widehat u_m^{\theta,\sigma,\gamma,\beta}
=
\bigl(h_m^{(\theta)}\bigr)^{-1}
\int_0^\infty
u(x)
\mathscr{L}_{m}^{(\theta,\sigma,\gamma,\beta)}(x)
\varrho^{\theta,\sigma,\gamma,\beta}(x)\,dx .
\end{equation}

To state the approximation estimate, we introduce a scale of weighted
Sobolev-type spaces generated by the scaled mapped derivative
\[
\mathscr{D}_{\theta,\sigma,\gamma,\beta}u
:=
x^{\frac{\gamma(\theta-\sigma)}{2}}
\mathscr{D}_{\gamma,\beta}
\left(
x^{-\frac{\gamma(\theta-\sigma)}{2}}u
\right),
\qquad
\mathscr{D}_{\gamma,\beta}u
:=
\frac{x^{1-\gamma}}{\gamma(\beta+1)}\partial_xu .
\]
Equivalently, \(\mathscr{D}_{\gamma,\beta}=d/dy\) with
\[
y=(\beta+1)x^\gamma .
\]
For \(\mu\in\mathbb{N}\), we define
\begin{equation}\label{eq:GFLF-Sobolev-space}
\mathcal{A}_{\theta,\sigma,\gamma,\beta}^{\mu}(\Lambda)
:=
\left\{
v\in L^2_{\varrho^{\theta,\sigma,\gamma,\beta}}(\Lambda):
\mathscr{D}_{\theta,\sigma,\gamma,\beta}^{\,r}v
\in L^2_{\varrho^{\theta+r,\sigma+r,\gamma,\beta}}(\Lambda),
\quad 1\le r\le \mu
\right\}.
\end{equation}
The associated seminorms and norm are given as
\[
|v|_{\mathcal{A}_{\theta,\sigma,\gamma,\beta}^{r}}
:=
\left\|
\mathscr{D}_{\theta,\sigma,\gamma,\beta}^{\,r}v
\right\|_{\varrho^{\theta+r,\sigma+r,\gamma,\beta}},
\qquad
0\le r\le \mu,
\]
and
\[
\|v\|_{\mathcal{A}_{\theta,\sigma,\gamma,\beta}^{\mu}}
:=
\left(
\sum_{r=0}^{\mu}
|v|_{\mathcal{A}_{\theta,\sigma,\gamma,\beta}^{r}}^2
\right)^{1/2}.
\]

\begin{theorem}\label{thm:GFLF-projection-estimate}
Let \(\mu,M\in\mathbb{N}\), \(0\le s\le\widehat\mu\), where
$\widehat\mu:=\min\{\mu,M+1\}.$
Let \(\theta>-1\), \(\sigma\in\mathbb{R}\), \(\beta>-1\), and \(0<\gamma<1\).
If \(u\in\mathcal{A}_{\theta,\sigma,\gamma,\beta}^{\mu}(\Lambda)\), then
\begin{equation}\label{eq:GFLF-projection-estimate}
\left\|
\mathscr{D}_{\theta,\sigma,\gamma,\beta}^{\,s}
\left(
u-\Pi_M^{\theta,\sigma,\gamma,\beta}u
\right)
\right\|_{\varrho^{\theta+s,\sigma+s,\gamma,\beta}}
\le
\left[
\frac{(M-\widehat\mu+1)!}{(M-s+1)!}
\right]^{1/2}
\left\|
\mathscr{D}_{\theta,\sigma,\gamma,\beta}^{\,\widehat\mu}u
\right\|_{\varrho^{\theta+\widehat\mu,\sigma+\widehat\mu,\gamma,\beta}} .
\end{equation}
\end{theorem}

\begin{proof}
Let
\[
u(x)
=
\sum_{m=0}^{\infty}
\widehat u_m^{\theta,\sigma,\gamma,\beta}
\mathscr{L}_{m}^{(\theta,\sigma,\gamma,\beta)}(x).
\]
From the differentiation identity \eqref{eq:GFLF-important-derivative},
repeated application gives
\begin{equation}\label{eq:GFLF-deriv-norm-expansion}
\mathscr{D}_{\theta,\sigma,\gamma,\beta}^{\,r}
\mathscr{L}_{m}^{(\theta,\sigma,\gamma,\beta)}(x)
=
(-1)^r
\mathscr{L}_{m-r}^{(\theta+r,\sigma+r,\gamma,\beta)}(x),
\qquad 0\le r\le m .
\end{equation}
By
\eqref{eq:GFLF-orth},
\[
\left\|
\mathscr{D}_{\theta,\sigma,\gamma,\beta}^{\,r}u
\right\|_{\varrho^{\theta+r,\sigma+r,\gamma,\beta}}^2
=
\sum_{m=r}^{\infty}
h_{m-r}^{(\theta+r)}
\left|
\widehat u_m^{\theta,\sigma,\gamma,\beta}
\right|^2,
\qquad r\ge1 .
\]
Consequently,
\[
\begin{aligned}
&
\left\|
\mathscr{D}_{\theta,\sigma,\gamma,\beta}^{\,s}
\left(
u-\Pi_M^{\theta,\sigma,\gamma,\beta}u
\right)
\right\|_{\varrho^{\theta+s,\sigma+s,\gamma,\beta}}^2
\\
&\qquad =
\sum_{m=M+1}^{\infty}
h_{m-s}^{(\theta+s)}
\left|
\widehat u_m^{\theta,\sigma,\gamma,\beta}
\right|^2
\\
&\qquad \le
\max_{m\ge M+1}
\frac{
h_{m-s}^{(\theta+s)}
}{
h_{m-\widehat\mu}^{(\theta+\widehat\mu)}
}
\sum_{m=M+1}^{\infty}
h_{m-\widehat\mu}^{(\theta+\widehat\mu)}
\left|
\widehat u_m^{\theta,\sigma,\gamma,\beta}
\right|^2
\\
&\qquad \le
\frac{
h_{M+1-s}^{(\theta+s)}
}{
h_{M+1-\widehat\mu}^{(\theta+\widehat\mu)}
}
\left\|
\mathscr{D}_{\theta,\sigma,\gamma,\beta}^{\,\widehat\mu}u
\right\|_{\varrho^{\theta+\widehat\mu,\sigma+\widehat\mu,\gamma,\beta}}^2 .
\end{aligned}
\]
Using
\[
h_m^{(\theta)}
=
\frac{\Gamma(m+\theta+1)}{\Gamma(m+1)},
\]
we obtain
\[
\frac{
h_{M+1-s}^{(\theta+s)}
}{
h_{M+1-\widehat\mu}^{(\theta+\widehat\mu)}
}
=
\frac{(M-\widehat\mu+1)!}{(M-s+1)!}.
\]
Substitution into the preceding estimate gives
\[
\left\|
\mathscr{D}_{\theta,\sigma,\gamma,\beta}^{\,s}
\left(
u-\Pi_M^{\theta,\sigma,\gamma,\beta}u
\right)
\right\|_{\varrho^{\theta+s,\sigma+s,\gamma,\beta}}^2
\le
\frac{(M-\widehat\mu+1)!}{(M-s+1)!}
\left\|
\mathscr{D}_{\theta,\sigma,\gamma,\beta}^{\,\widehat\mu}u
\right\|_{\varrho^{\theta+\widehat\mu,\sigma+\widehat\mu,\gamma,\beta}}^2 .
\]
Taking square roots proves \eqref{eq:GFLF-projection-estimate}.
\end{proof}

\subsection{Interpolation estimate}
\label{subsec:GFLF-interpolation}

Let \(\{y_i^{(\theta)}\}_{i=0}^{M}\) be the Laguerre--Gauss nodes associated
with the classical Laguerre weight \(y^\theta e^{-y}\) on \((0,\infty)\). The
corresponding scaled fractional nodes are
\[
x_i^{(\theta,\gamma,\beta)}
:=
\left(\frac{y_i^{(\theta)}}{\beta+1}\right)^{1/\gamma},
\qquad 0\le i\le M .
\]
Set
\[
\eta_{\theta,\sigma,\gamma}:=\frac{\gamma(\theta-\sigma)}{2}.
\]
We define the generalized interpolation operator
\[
\mathcal{I}_{M}^{\theta,\sigma,\gamma,\beta}:
C(\Lambda)\longrightarrow
\mathbb{P}_{M}^{\eta_{\theta,\sigma,\gamma},\gamma}
\]
by the nodal interpolation conditions
\[
\mathcal{I}_{M}^{\theta,\sigma,\gamma,\beta}v
\!\left(x_i^{(\theta,\gamma,\beta)}\right)
=
v\!\left(x_i^{(\theta,\gamma,\beta)}\right),
\qquad 0\le i\le M .
\]
Here
\[
\mathbb{P}_{M}^{\eta_{\theta,\sigma,\gamma},\gamma}
:=
\left\{
x^{\eta_{\theta,\sigma,\gamma}}q(x):
q\in\mathbb{P}^\gamma_{M}
\right\}.
\]

Equivalently,
\begin{equation}\label{eq:GFLF-interpolation-formula}
\mathcal{I}_{M}^{\theta,\sigma,\gamma,\beta}v(x)
=
\sum_{i=0}^{M}
v\!\left(x_i^{(\theta,\gamma,\beta)}\right)
\ell_i^{(\theta,\sigma,\gamma,\beta)}
\!\left(Y_{\gamma,\beta}(x)\right),
\qquad
Y_{\gamma,\beta}(x):=(\beta+1)x^\gamma ,
\end{equation}
where the generalized scaled mapped Lagrange functions are given by
\[
\ell_i^{(\theta,\sigma,\gamma,\beta)}
\!\left(Y_{\gamma,\beta}(x)\right)
=
\frac{
x^{\eta_{\theta,\sigma,\gamma}}
\displaystyle
\prod_{\substack{r=0\\ r\ne i}}^{M}
\left(
x^\gamma-
\left(x_r^{(\theta,\gamma,\beta)}\right)^\gamma
\right)
}{
\left(x_i^{(\theta,\gamma,\beta)}\right)^{\eta_{\theta,\sigma,\gamma}}
\displaystyle
\prod_{\substack{r=0\\ r\ne i}}^{M}
\left(
\left(x_i^{(\theta,\gamma,\beta)}\right)^\gamma
-
\left(x_r^{(\theta,\gamma,\beta)}\right)^\gamma
\right)
}.
\]
Hence,
\[
\ell_i^{(\theta,\sigma,\gamma,\beta)}
\!\left(Y_{\gamma,\beta}(x_j^{(\theta,\gamma,\beta)})\right)
=
\delta_{ij},
\qquad 0\le i,j\le M .
\]

The operator \(\mathcal{I}_{M}^{\theta,\sigma,\gamma,\beta}\) can be written in
terms of the scaled fractional interpolation operator
\(\mathcal{I}_{M}^{\theta,\gamma,\beta}\) as
\begin{equation}\label{eq:GFLF-interp-relation}
\mathcal{I}_{M}^{\theta,\sigma,\gamma,\beta}v(x)
=
x^{\eta_{\theta,\sigma,\gamma}}
\mathcal{I}_{M}^{\theta,\gamma,\beta}
\left\{
x^{-\eta_{\theta,\sigma,\gamma}}v(x)
\right\}.
\end{equation}
Consequently,
\[
\mathcal{I}_{M}^{\theta,\sigma,\gamma,\beta}v
\in
\mathbb{P}_{M}^{\eta_{\theta,\sigma,\gamma},\gamma}.
\]

\begin{theorem}\label{thm:GFLF-interpolation-estimate}
Let \(\mu,M\in\mathbb{N}\), \(\theta>-1\), \(\sigma\in\mathbb{R}\),
\(\beta>-1\), \(0<\gamma<1\), and
$\widehat\mu:=\min\{\mu,M+1\}.$
Assume that
$v\in C(\Lambda)\cap \mathcal{A}_{\theta,\sigma,\gamma,\beta}^{\mu}(\Lambda),
\
\mathscr{D}_{\theta,\sigma,\gamma,\beta}v
\in
\mathcal{A}_{\theta,\sigma,\gamma,\beta}^{\mu-1}(\Lambda).$
Then
\begin{equation}\label{eq:GFLF-interpolation-estimate}
\left\|
\mathcal{I}_{M}^{\theta,\sigma,\gamma,\beta}v-v
\right\|_{\varrho^{\theta,\sigma,\gamma,\beta}}
\le
C
\left[
\frac{(M+1-\widehat\mu)!}
{M!}
\right]^{1/2}
\left[
\left\|
\mathscr{D}_{\theta,\sigma,\gamma,\beta}^{\,\widehat\mu}v
\right\|_{\varrho^{\theta+\mu-1,\sigma+\mu-1,\gamma,\beta}}
+
2\sqrt{\log M}
\left\|
\mathscr{D}_{\theta,\sigma,\gamma,\beta}^{\,\widehat\mu}v
\right\|_{\varrho^{\theta+\mu,\sigma+\mu,\gamma,\beta}}
\right].
\end{equation}
\end{theorem}

\begin{proof}
Define
\[
z(x):=x^{-\eta_{\theta,\sigma,\gamma}}v(x).
\]
From \eqref{eq:GFLF-interp-relation}, we obtain
\[
\mathcal{I}_{M}^{\theta,\sigma,\gamma,\beta}v(x)-v(x)
=
x^{\eta_{\theta,\sigma,\gamma}}
\left(
\mathcal{I}_{M}^{\theta,\gamma,\beta}z(x)-z(x)
\right).
\]
Therefore,
\[
\left\|
\mathcal{I}_{M}^{\theta,\sigma,\gamma,\beta}v-v
\right\|_{\varrho^{\theta,\sigma,\gamma,\beta}}
=
\left\|
\mathcal{I}_{M}^{\theta,\gamma,\beta}z-z
\right\|_{\varrho^{\theta,\theta,\gamma,\beta}} .
\]
Indeed,
\[
x^{2\eta_{\theta,\sigma,\gamma}}\varrho^{\theta,\sigma,\gamma,\beta}(x)
=
x^{\gamma(\theta-\sigma)}
\gamma(\beta+1)^{\theta+1}
x^{\gamma(\sigma+1)-1}e^{-(\beta+1)x^\gamma}
=
\gamma(\beta+1)^{\theta+1}
x^{\gamma(\theta+1)-1}e^{-(\beta+1)x^\gamma}
=
\varrho^{\theta,\theta,\gamma,\beta}(x).
\]

Moreover, by the definition of
\(\mathscr{D}_{\theta,\sigma,\gamma,\beta}\),
\[
\mathscr{D}_{\gamma,\beta}z
=
x^{-\eta_{\theta,\sigma,\gamma}}
\mathscr{D}_{\theta,\sigma,\gamma,\beta}v .
\]
Repeated application gives
\[
\mathscr{D}_{\gamma,\beta}^{\,\widehat\mu}z
=
x^{-\eta_{\theta,\sigma,\gamma}}
\mathscr{D}_{\theta,\sigma,\gamma,\beta}^{\,\widehat\mu}v .
\]
Consequently,
\[
\left\|
\mathscr{D}_{\gamma,\beta}^{\,\widehat\mu}z
\right\|_{\varrho^{\theta+\mu-1,\theta+\mu-1,\gamma,\beta}}
=
\left\|
\mathscr{D}_{\theta,\sigma,\gamma,\beta}^{\,\widehat\mu}v
\right\|_{\varrho^{\theta+\mu-1,\sigma+\mu-1,\gamma,\beta}},
\]
and
\[
\left\|
\mathscr{D}_{\gamma,\beta}^{\,\widehat\mu}z
\right\|_{\varrho^{\theta+\mu,\theta+\mu,\gamma,\beta}}
=
\left\|
\mathscr{D}_{\theta,\sigma,\gamma,\beta}^{\,\widehat\mu}v
\right\|_{\varrho^{\theta+\mu,\sigma+\mu,\gamma,\beta}} .
\]
Applying Theorem~\ref{thm:FLF-interpolation-error} to \(z\), with the scaled
fractional variable \(Y_{\gamma,\beta}(x)=(\beta+1)x^\gamma\), gives
\eqref{eq:GFLF-interpolation-estimate}.
\end{proof}

\section{Numerical results}

\subsection{Approximation by fractional Laguerre functions}

In this subsection, we present numerical experiments illustrating the
approximation behavior of the fractional Laguerre functions on the
semi-infinite interval. We first examine the distribution of the transformed
Laguerre points. Figure~\ref{fig:FLF-nodes-all} contains three panels. The
left panel shows the effect of the scaling parameter \(\beta\), with fixed
\(\theta=0\), \(\gamma=1/3\), and \(M=80\). The middle panel displays the
influence of the Laguerre parameter \(\theta\), with fixed \(M=80\),
\(\gamma=1/3\), and \(\beta=20\). The right panel illustrates the effect of
the fractional parameter \(\gamma\), with fixed \(M=80\), \(\theta=0\), and
\(\beta=20\). These plots demonstrate how the parameters \(\beta\), \(\theta\),
and \(\gamma\) affect the location of the collocation points and their
concentration in the physical variable.

\begin{figure}[htbp]
\centering
\includegraphics[width=1\textwidth]{\detokenize{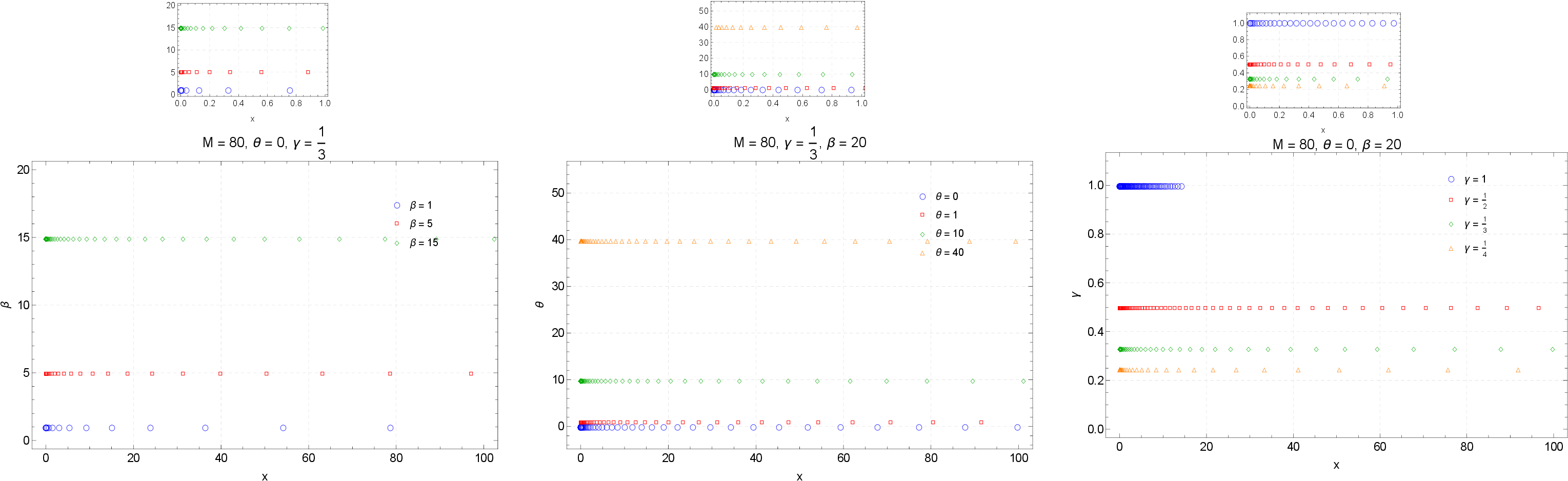}}
\caption{
Fractional Laguerre point distributions. Left: effect of \(\beta\) for
\(\theta=0\), \(\gamma=1/3\), and \(M=80\). Middle: effect of
\(\theta\) for \(M=80\), \(\gamma=1/3\), and \(\beta=20\). Right: effect
of \(\gamma\) for \(M=80\), \(\theta=0\), and \(\beta=20\).
}
\label{fig:FLF-nodes-all}
\end{figure}

We next investigate the projection accuracy for functions whose dominant
regularity is expressed in fractional powers of \(x\). The test functions are
\[
        u_1(x)=\sin(x^\alpha),\qquad
        u_2(x)=\exp(-x^\alpha),\qquad
        u_3(x)=x^\alpha E_{\alpha,\alpha+1}(-x^\alpha),
\]
where \(E_{\alpha,\alpha+1}\) denotes the two-parameter Mittag--Leffler
function. In each experiment, \(\alpha\) denotes the fractional scale appearing
in the target function, whereas \(\gamma\) denotes the fractional parameter used
in the approximation basis.

In the first projection experiment, we take \(\alpha=1/3\). Figure
\(\ref{fig:FLF-proj-gamma-third}\) reports the projection errors for
\(\theta=0\), \(\beta=20\), and approximation parameters
\(\gamma=1,1/6,1/3\). The results show that the fractional Laguerre basis is
more effective when the fractional scale in the approximation space is
compatible with the algebraic structure of the target functions.

\begin{figure}[htbp]
\centering
\includegraphics[width=1\textwidth]{\detokenize{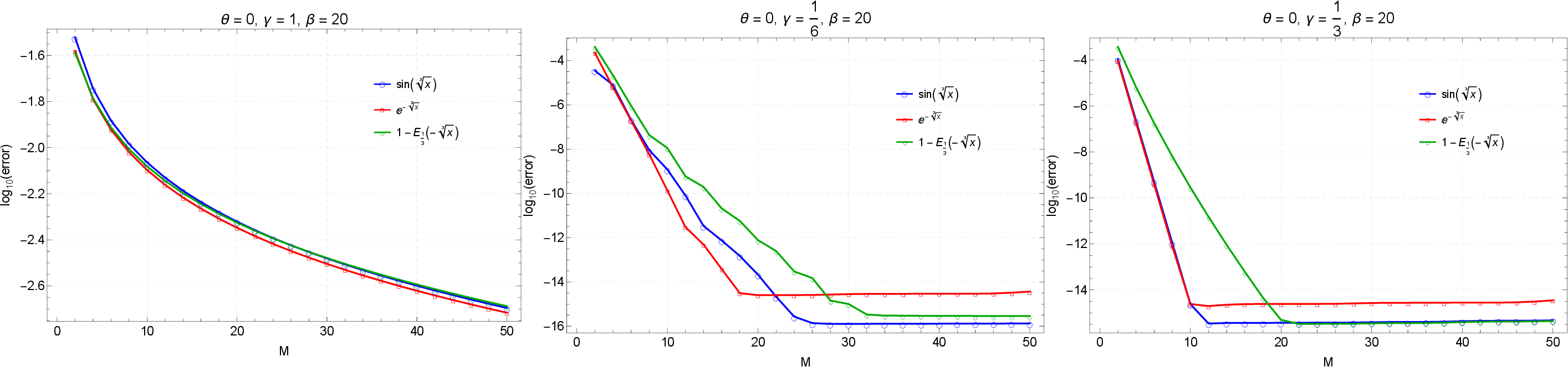}}
\caption{
Projection errors for
\(u_1(x)=\sin(x^{1/3})\), \(u_2(x)=\exp(-x^{1/3})\), and
\(u_3(x)=x^{1/3}E_{1/3,4/3}(-x^{1/3})\), with
\(\theta=0\), \(\beta=20\), and approximation parameters
\(\gamma=1,1/6,1/3\).
}
\label{fig:FLF-proj-gamma-third}
\end{figure}

The second projection experiment studies the influence of the scaling
parameter \(\beta\). We consider target functions with fractional scale
\(\alpha=2/3\). Figure~\ref{fig:FLF-proj-beta} shows the errors for fixed
\(\theta=0\) and \(\gamma=2/3\), using \(\beta=1,5,20\). The parameter
\(\beta\) changes the effective distribution of the Laguerre points in the
physical variable and therefore has a visible effect on the projection
accuracy.

\begin{figure}[htbp]
\centering
\includegraphics[width=1\textwidth]{\detokenize{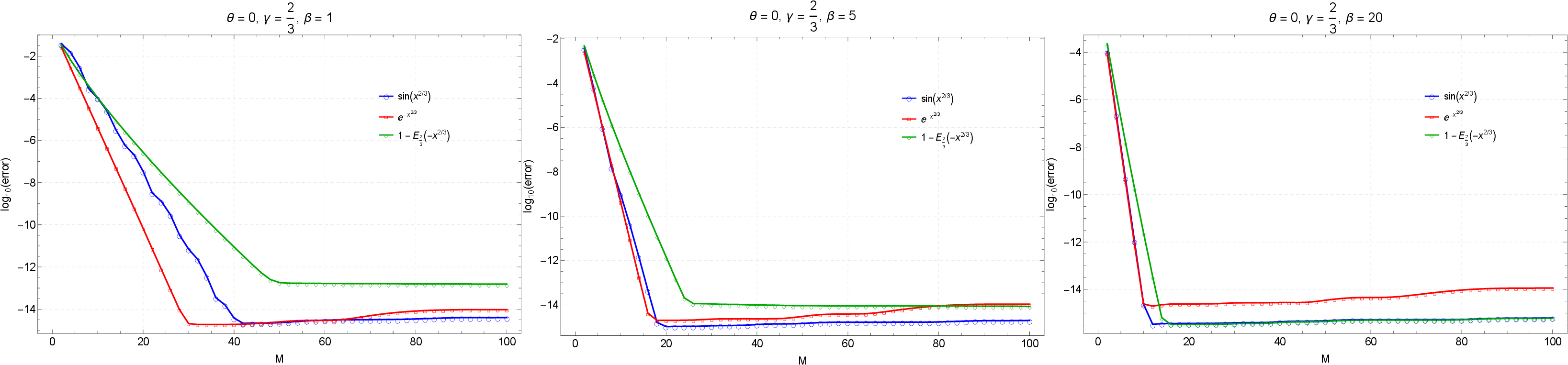}}
\caption{
Projection errors for
\(u_1(x)=\sin(x^{2/3})\), \(u_2(x)=\exp(-x^{2/3})\), and
\(u_3(x)=x^{2/3}E_{2/3,5/3}(-x^{2/3})\), with
\(\theta=0\), \(\gamma=2/3\), and \(\beta=1,5,20\).
}
\label{fig:FLF-proj-beta}
\end{figure}

Finally, we consider the weaker algebraic scale \(\alpha=1/4\).
Figure~\ref{fig:FLF-proj-gamma-fourth} gives the projection errors for
\(\theta=0\), \(\beta=20\), and approximation parameters
\(\gamma=1,1/4,1/8\). The comparison confirms the sensitivity of the
approximation to the choice of the fractional parameter. In particular,
selecting \(\gamma\) in accordance with the dominant fractional power of the
target functions significantly improves the decay of the projection error.

\begin{figure}[htbp]
\centering
\includegraphics[width=1\textwidth]{\detokenize{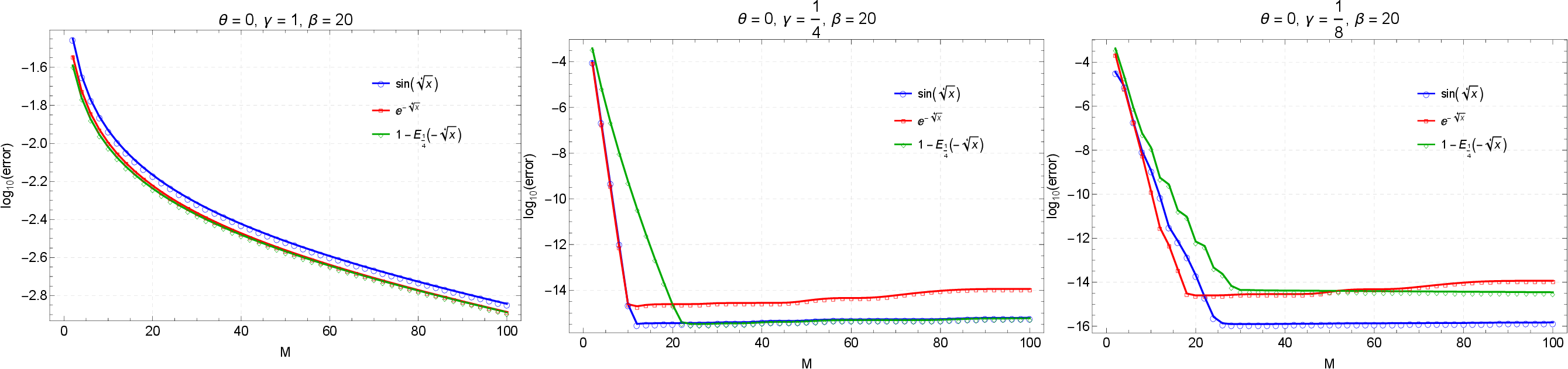}}
\caption{
Projection errors for
\(u_1(x)=\sin(x^{1/4})\), \(u_2(x)=\exp(-x^{1/4})\), and
\(u_3(x)=x^{1/4}E_{1/4,5/4}(-x^{1/4})\), with
\(\theta=0\), \(\beta=20\), and approximation parameters
\(\gamma=1,1/4,1/8\).
}
\label{fig:FLF-proj-gamma-fourth}
\end{figure}

\subsection{Approximation by generalized fractional Laguerre functions}

In this subsection, we examine the numerical performance of the generalized
fractional Laguerre expansion. The purpose is to illustrate the role of the
additional algebraic parameter \(\sigma\) and the scaling parameter \(\beta\)
in approximating functions with endpoint singularities or algebraic growth.

We first study the effect of \(\beta\). The three test functions are
\[
        u_1(x)=\frac{\sin x}{x},\qquad
        u_2(x)=\frac{1}{\sqrt{x}},\qquad
        u_3(x)=\sqrt{x}\,e^{\sqrt{x}} .
\]
For \(u_1\), we take \(\theta=0\), \(\sigma=2\), and \(\gamma=1\), and compare
\(\beta=1,16,24\). For \(u_2\), we take \(\theta=0\), \(\sigma=2\), and
\(\gamma=1/2\), and compare \(\beta=1,8,16\). For \(u_3\), we take
\(\theta=2\), \(\sigma=0\), and \(\gamma=1/2\), and compare
\(\beta=4,16,24\). The corresponding projection errors are shown in
Fig.~\ref{fig:GFLF-beta-effect}. The curves indicate that the scaling
parameter \(\beta\) influences the rate at which the projection error decreases,
because it modifies the effective distribution of the Laguerre points in the
physical variable and changes the associated weighted structure.

\begin{figure}[htbp]
\centering
\includegraphics[width=1\textwidth]{\detokenize{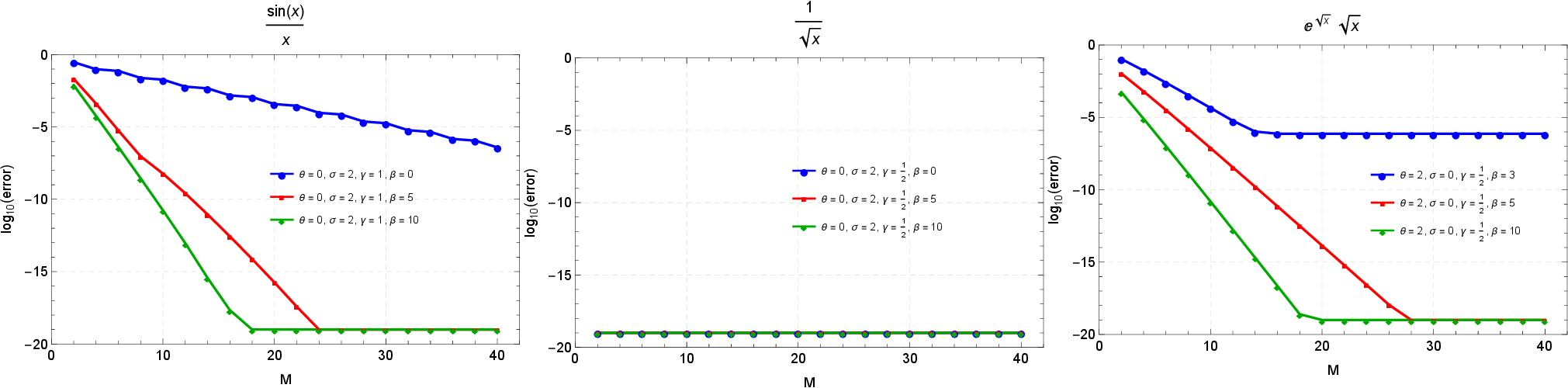}}
\caption{
Projection errors for the generalized fractional Laguerre approximation,
showing the effect of the scaling parameter \(\beta\). Left:
\(u_1(x)=\sin(x)/x\), with \(\theta=0\), \(\sigma=2\), \(\gamma=1\), and
\(\beta=1,16,24\). Middle: \(u_2(x)=1/\sqrt{x}\), with
\(\theta=0\), \(\sigma=2\), \(\gamma=1/2\), and \(\beta=1,8,16\). Right:
\(u_3(x)=\sqrt{x}e^{\sqrt{x}}\), with \(\theta=2\), \(\sigma=0\),
\(\gamma=1/2\), and \(\beta=4,16,24\).
}
\label{fig:GFLF-beta-effect}
\end{figure}

We next examine the effect of the fractional parameter \(\gamma\). The test
functions are
\[
        u_1(x)=x^{-\gamma}\sin(x^\gamma),\qquad
        u_2(x)=x^{-\gamma},\qquad
        u_3(x)=x^\gamma e^{x^\gamma}.
\]
In each panel, the same value of \(\gamma\) is used in the target function and
in the approximation basis. For \(u_1\), we fix \(\theta=0\), \(\sigma=2\),
and \(\beta=16\), and compare \(\gamma=1/4,1/2,1\). For \(u_2\), we fix
\(\theta=0\), \(\sigma=2\), and \(\beta=8\), and compare
\(\gamma=1/4,1/2,1\). For \(u_3\), we fix \(\theta=2\), \(\sigma=0\), and
\(\beta=16\), and compare \(\gamma=1/2,2/3,1\). The corresponding results are
displayed in Fig.~\ref{fig:GFLF-gamma-effect}. The observed error decay shows
that matching the fractional structure of the basis with that of the target
function improves the accuracy of the approximation.

\begin{figure}[htbp]
\centering
\includegraphics[width=1\textwidth]{\detokenize{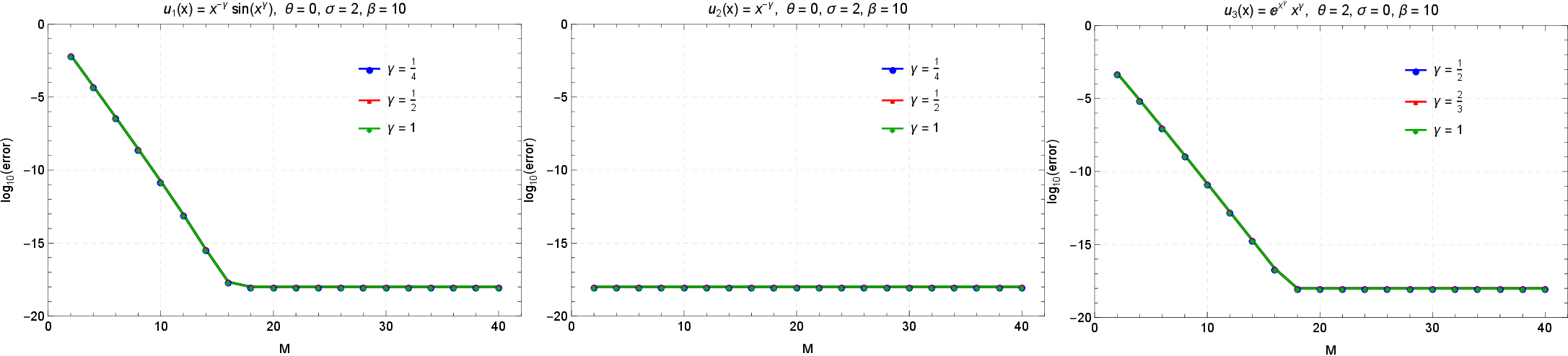}}
\caption{
Projection errors for the generalized fractional Laguerre approximation,
showing the effect of the fractional parameter \(\gamma\). Left:
\(u_1(x)=x^{-\gamma}\sin(x^\gamma)\), with \(\theta=0\), \(\sigma=2\),
\(\beta=16\), and \(\gamma=1/4,1/2,1\). Middle:
\(u_2(x)=x^{-\gamma}\), with \(\theta=0\), \(\sigma=2\), \(\beta=8\),
and \(\gamma=1/4,1/2,1\). Right:
\(u_3(x)=x^\gamma e^{x^\gamma}\), with \(\theta=2\), \(\sigma=0\),
\(\beta=16\), and \(\gamma=1/2,2/3,1\).
}
\label{fig:GFLF-gamma-effect}
\end{figure}

\appendix

\section{Auxiliary identities for Laguerre polynomials}
\label{app:Laguerre-properties}

For completeness, we collect several standard identities for generalized
Laguerre polynomials that are used throughout the paper. Let \(\theta>-1\),
and let \(L_m^{(\theta)}\) denote the generalized Laguerre polynomial of
degree \(m\) on \(\mathbb{R}^{+}\).

The first two polynomials are
\[
L_{0}^{(\theta)}(y)=1,
\qquad
L_{1}^{(\theta)}(y)=-y+\theta+1 .
\]
For \(m\ge1\), the three-term recurrence relation is
\begin{equation}\label{eq:Laguerre-recurrence}
L_{m+1}^{(\theta)}(y)
=
\frac{2m+\theta+1-y}{m+1}L_{m}^{(\theta)}(y)
-
\frac{m+\theta}{m+1}L_{m-1}^{(\theta)}(y).
\end{equation}

The polynomial \(L_m^{(\theta)}\) satisfies the singular Sturm--Liouville
equation
\begin{equation}\label{eq:Laguerre-SL}
y^{-\theta}e^{y}
\partial_{y}
\left(
y^{\theta+1}e^{-y}\partial_{y}L_{m}^{(\theta)}(y)
\right)
+
mL_{m}^{(\theta)}(y)=0 .
\end{equation}

We shall also use the following derivative and connection identities:
\begin{align}
L_{m}^{(\theta)}(y)
&=
\partial_{y}L_{m}^{(\theta)}(y)
-
\partial_{y}L_{m+1}^{(\theta)}(y),
\label{eq:Laguerre-connection-deriv}\\
y\partial_{y}L_{m}^{(\theta)}(y)
&=
mL_{m}^{(\theta)}(y)
-
(m+\theta)L_{m-1}^{(\theta)}(y),
\qquad m\ge1,
\label{eq:Laguerre-y-deriv}\\
\partial_{y}L_{m}^{(\theta)}(y)
&=
-L_{m-1}^{(\theta+1)}(y)
=
-\sum_{r=0}^{m-1}L_{r}^{(\theta)}(y),
\qquad m\ge1 .
\label{eq:Laguerre-derivative}
\end{align}

Finally, we recall the Laguerre--Gauss quadrature formula. Let
\(\{y_i^{(\theta)}\}_{i=0}^{M}\) be the zeros of
\(L_{M+1}^{(\theta)}(y)\). The corresponding quadrature weights are
\begin{equation}\label{eq:Laguerre-Gauss-weights}
\varpi_i^{(\theta)}
=
\frac{\Gamma(M+\theta+1)}
{(M+\theta+1)(M+1)!}
\frac{y_i^{(\theta)}}
{\left[L_{M}^{(\theta)}(y_i^{(\theta)})\right]^2},
\qquad 0\le i\le M .
\end{equation}
Hence,
\begin{equation}\label{eq:Laguerre-Gauss-quadrature}
\int_{0}^{\infty} q(y)y^{\theta}e^{-y}\,dy
=
\sum_{i=0}^{M}
q(y_i^{(\theta)})\varpi_i^{(\theta)},
\qquad
\forall q\in \mathbb{P}_{2M+1}.
\end{equation}

\section*{Data availability statement}
No datasets were generated or analyzed during the current study.

\section*{Declarations}

\section*{Conflict of interest}
The authors declare that they have no conflict of interest.

 \section*{Author Contributions}   Author Contributions All authors contributed equally to this work.

	\section*{Funding}This work was supported and funded by the Deanship of Scientific Research at Imam Mohammad Ibn Saud Islamic University (IMSIU).

\bibliographystyle{elsart-num-sort}
		
		
		\bibliography{Bibfileamc}


\end{document}